\documentclass{amsart}
\usepackage{amsfonts,amssymb,amsmath,amsthm}
\usepackage{url}
\usepackage{enumerate}

\usepackage{amsmath,amsfonts,amsthm,url,color,amssymb, mathtools}
\usepackage{graphicx}
\usepackage[all]{xy}
\usepackage{graphicx}
\usepackage{algpseudocode, algorithm}
\usepackage{hyperref}

\newcommand{\Tr}{\mathrm{Tr}}
\newcommand{\fqn}{\mathbb{F}_{q^n}}
\newcommand{\F}{\mathbb{F}}
\newcommand{\fq}{\mathbb{F}_q}

\newcommand{\li}{\mathcal L}
\newcommand{\GF}[1]{\mathbb{F}_{#1}}
\newcommand{\order}[1]{\left|#1\right|}

\newtheorem{theorem}{Theorem}[section]
\newtheorem{lemma}[theorem]{Lemma}

\newtheorem{corollary}[theorem]{Corollary}
\newtheorem{proposition}[theorem]{Proposition}
\theoremstyle{definition}
\newtheorem{definition}[theorem]{Definition}

\theoremstyle{remark}
\usepackage{amssymb}
\newtheorem{remark}[theorem]{Remark}
\usepackage{enumerate}
\author[L. Reis]{Lucas Reis}
\address{Departamento de Matem\'{a}tica, Universidade Federal de Minas Gerais, Belo Horizonte, Minas Gerais, 31270-901, Brazil}
\email{lucasreismat@gmail.com}
\author[Q. Wang]{Qiang Wang}
\address{School of Mathematics and Statistics, Carleton University, Ottawa, Ontario, K1S 5B6, Canada}
\thanks{The research of the author Qiang Wang is partially supported by NSERC of Canada (RGPIN- 2017-06410) and the research of the author Lucas Reis is partially supported by PRPq/UFMG (ADRC 09/2019).}
\email{wang@math.carleton.ca}

\begin{document}

\title{Permutation polynomials from a linearized decomposition} 

\begin{abstract}
In this paper we discuss the permutational property of polynomials of the form $f(L(x))+k(L(x))\cdot M(x)\in \F_{q^n}[x]$ over the finite field $\F_{q^n}$, where $L, M\in \F_q[x]$ are $q$-linearized polynomials. The restriction $L, M\in \F_q[x]$ implies a nice correspondence between the pair $(L, M)$ and the pair $(g, h)$ of conventional $q$-associates over $\F_q$ of degree at most $n-1$. In particular, by using the AGW criterion, permutational properties of our class of polynomials translates to some arithmetic properties of polynomials over $\F_q$, like coprimality.  This relates the problem of constructing PPs of $\fqn$ to the problem of  factorizing  $x^n-1$ in $\fq[x]$.  We then specialize to the case where $L(x)$ is the trace polynomial from $\F_{q^n}$ over $\F_q$, providing results on the construction of permutation and complete permutation polynomials, and their inverses. We further demonstrate that the latter can be extended to more general linearized polynomials of degree $q^{n-1}$.
\end{abstract}

\maketitle
\section{Introduction}


Let $q = p^k$ be a power of a prime number $p$, let $\GF{q}$ be a finite field with $q$ elements, and let $\GF{q}[x]$ be the ring of polynomials over $\GF{q}$. We call $f(x) \in \GF{q}[x]$ a \emph{permutation polynomial} (PP) over $\GF{q}$ if its associated polynomial mapping $f:c\mapsto f(c)$ from $\GF{q}$ to itself is a bijection. It is well known that every permutation of $\GF{q}$ can be expressed as a permutation polynomial over $\GF{q}$, of degree at most $q-1$.

Permutation polynomials over finite fields have been a hot topic of study for many years, partially due to their applications in coding theory \cite{TIT:DingH13,FFA:Laigle07}, cryptography\cite{ CRYPTO:LidlM83, CACM:RivestSA78, EL:SchwenkH98}, combinatorial designs \cite{JCT:DingY06}, and other areas of mathematics and engineering.  More background material and information about properties, constructions, and applications of permutation polynomials may be found in \cite[Chapter 7]{FFields:LidlN97} and \cite[Chapter 8]{HandbookFFields:MullenP13}. For a detailed survey of open questions and recent results we refer the reader to  \cite{FFA:Hou15} and \cite{IndexSurvey}.

Recently, Akbary, Ghioca and Wang have derived the following useful criterion to study permutation functions on finite sets. It first appears in \cite{AGW} and  is further developed in \cite{DCC:LiQW17, FFA:YuanD11, FFA:YuanD14, DCC:ZhengYP16},  among others.

\begin{lemma}[The AGW Criterion]\label{lemma:AGWCriterion}
Let $A$, $S$ and $\bar{S}$ be finite sets with $\order{S} = \order{\bar{S}}$, and let $f:A\rightarrow A$, $\bar{f}:S\rightarrow \bar{S}$, $\lambda:A\rightarrow S$, and $\bar{\lambda}:A\rightarrow \bar{S}$  be maps such that $\bar{\lambda} \circ f=\bar{f} \circ\lambda$ (see the following commutative diagram).
\[
\xymatrix{
A \ar[r]^{f}\ar[d]^{\lambda} & A \ar[d]^{\bar{\lambda}}\\
S \ar[r]^{\bar{f}} & \bar{S} }
\]

If both $\lambda$ and $\bar{\lambda}$ are surjective, then the following statements are equivalent:
\begin{itemize}
\item $f$ is a bijection from $A$ to $A$ (a permutation over $A$);
\item $\bar{f}$ is a bijection from $S$ to $\bar{S}$ and $f$ is injective on $\lambda^{-1}(s)$ for each $s\in S$.
\end{itemize}
\end{lemma}

The importance of the AGW criterion depends on that it can be used not only to explain some previous constructions of PPs, but also to construct numerous new classes. For example,  Akbary, Ghoica and Wang  \cite{AGW}  applied their approach into different cases (e.g., multiplicative group case, elliptic curve case, additive group case) and obtained many interesting results.
In the additive group case,
for any polynomial  $g \in  \fqn[x]$, any additive polynomials $\varphi, \psi,\bar{\psi}\in\fqn[x]$ satisfying $\varphi \circ \psi = \bar{\psi} \circ \varphi$ and $\#\psi(\fqn)=\#\bar{\psi}(\fqn)$, and any polynomial $h \in \fqn[x]$ such that $h(\psi(\F_{q^n}))\subseteq \F_q^*$,   the permutation polynomials of the form $f(x) := h(\psi(x)) \varphi (x) + g(\psi(x))$ over $\F_{q^n}$ were characterized. One of two necessary and sufficient conditions requires $\ker(\varphi) \cap \ker(\psi) = \{0\}$, equivalently, $\varphi$ induces a bijection betwenen $\ker(\psi)$ and $\ker(\bar{\psi})$.  Later on, Yuan and Ding \cite{FFA:YuanD11}  extended their study to  PPs with the  form  $f(x)=g(B(x))+\sum_{i=1}^{r}\left(L_i(x)+\delta_i\right)h_i(B(x))$ over $\fqn$, where $B(x), L_1(x), \ldots, L_r(x) \in \fq[x]$ are $q$-polynomials,  $g(x) \in \fqn[x]$, $h_1(x), \ldots, h_r(x) \in \fq[x]$, and $\delta_1, \ldots, \delta_r \in \fqn$ such that $B(\delta_i) \in \fq$ and $h_i(B(\fqn))\subseteq \fq$.  In this case, the condition $\ker(B) \cap \ker(\sum_{i=1}^r L_i h_i(y)) = \{0\}$ for each $y \in B(\fqn)$ reduced to  $\gcd(\sum_{i=1}^r l_i(x) h_i(y), b(x)) =1$ for any $y \in \fq$, where $l_i(x)$ and $b(x)$ are conventional $q$-associate of $L_i(x)$ and $B(x)$. 
 Several interesting classes of PPs of the form $L(x) +  g(x^q -x + \delta) \in \mathbb{F}_{q^n}[x]$, where $L(x) $ is a linearized polynomial and  $g(x)^q = g(x)$ were also given in \cite{FFA:YuanD14}.  
Further generic applications of AGW criterion over $\F_{q^2}$ can be found in Zheng, Yuan and Pei \cite{DCC:ZhengYP16} 
and   Li, Qu and Wang \cite{DCC:LiQW17}.   





In this paper we focus on the subclass of permutation polynomials of $\F_{q^n}$ with the form \begin{equation}\label{eq:main}P(x)=f(L(x))+k(L(x))\cdot M(x),\end{equation} where $f\in \F_{q^n}[x]$, and $L, M\in \fqn[x]$ are $q$-linearized polynomials with coefficients in $\fq$.  We further relate the problem of constructing PPs of $\fqn$ to the problem of  factorizing  $x^n-1$ in $\fq[x]$.  Using explicit factors of $x^n-1$ in $\fq[x]$,  we can construct PPs of $\fqn$  by taking their linearized $q$-associates for $L(x)$.   In the special case that  $L(x)=\Tr_{q^n/q}(x):=x^{q^{n-1}}+\cdots+x$ is the trace polynomial, we provide more explicit results concerning the construction of permutations and their inverses, and also the construction of complete permutation polynomials (i.e., complete mappings). Most notably, Theorem~\ref{thm:const} provides a general method to produce permutations of $\F_{q^n}$ from permutations of $\F_q$, by simply solving a system of equations of the form $\Tr_{q^n/q}(x_i)=y_i$. In fact, the same method can be applied to construct complete mappings of $\F_{q^n}$ from complete mappings of $\F_q$. We further show that the results on the trace case are extended to polynomials $L(x)=\sum_{i=0}^{n-1}a^ix^{q^{n-1-i}}$ with $a\in \F_q$ and $a^n=1$.

It is worthy of mention that past works have considered permutation polynomials like in Eq.~\eqref{eq:main} with $L(x)$ being the trace map~\cite{LWZ, TZJ, WY, ZTT}. In some cases, the polynomial $M$ can even be replaced by a more general one, but the permutation criteria become less explicit.
Nevertheless, the polynomial $f$ is rather too restricted: in particular, $f$ has at most $3$ nonzero coefficients. 

The paper is organized as follows. In Section 2 we provide background material and present some initial results, including a criterion for when a polynomial given by Eq.~\eqref{eq:main} permutes $\F_{q^n}$. In Section 3 we specialize our study to permutations arising from the trace function and discuss the construction of permutation and complete permutation polynomials, and their inverses. Finally, in Section 4, we provide concluding remarks and propose directions for future research.

\section{Some preliminary results}
Throughout this paper, $\F_q$ denotes the finite field with $q$ elements and $\overline{\F}_q$ denotes its algebraic closure. A $q$-linearized polynomial is a polynomial of the form $\sum_{i=0}^ma_ix^{q^i}$, where $a_i\in \overline{\F}_q$. From the well-known identity $(a+b)^q=a^q+b^q$, we observe that $q$-linearized polynomials with coefficients in $\F_{q^t}$ induce $\F_q$-linear maps over every finite extension of $\F_{q^t}$. 

\subsection{Background material}
\begin{definition}
For $f\in \F_q[x]$ with $f(x)=\sum_{i=0}^ma_ix^i$, the linearized $q$-associate of $f$ is the polynomial $L_f(x)=\sum_{i=0}^ma_ix^{q^i}$.
\end{definition}
The following lemma provides some properties of the $q$-associates through basic operations. Its proof follows by direct verification so we omit details.

\begin{lemma}\label{lem:q-associate}
For $f, g\in \F_q[x]$, we have that $L_{f+g}(x)=L_f(x)+L_g(x)$, $L_f(L_g(x))=L_{fg}(x)$ and $\gcd(L_f(x), L_g(x))=L_{\gcd(f, g)}(x)$. 
\end{lemma}

We obtain the following corollary.

\begin{corollary}\label{cor:roots}
For a positive integer $n$ and a nonzero polynomial $f\in \F_q[x]$, the equation $L_f(y)=0$ has $q^{\deg(F)}$ solutions over $\F_{q^n}$, where $F(x)=\gcd(f(x), x^n-1)$. Moreover, such solutions comprise an $\F_q$-vector space and, for a monic divisor $g\in \F_q[x]$ of $x^n-1$, the image set $L_g(\F_{q^n})$ equals the set of roots of $L_G(x)$, where $G(x)=\frac{x^n-1}{g(x)}$.
\end{corollary}

\begin{proof}
We observe that, for $h(x)=x^n-1$, the roots of $L_{h}(x)=x^{q^n}-x$ are simple and comprise the field $\F_{q^n}$.  From Lemma~\ref{lem:q-associate}, we have that $\gcd(L_f(x), x^{q^n}-x)=L_F(x)$. In particular, $L_F$ splits completely over $\F_{q^n}$ into distinct linear factors. Therefore, $L_f(y)=0$ has $\deg(L_F)=q^{\deg(F)}$ solutions over $\F_{q^n}$. As $L_F$ is $q$-linearized, the evaluation map $c\mapsto L_f(c)$ is an $\F_q$-linear map over $\F_{q^n}$. Therefore, the set of solutions of $L_f(y)=0$ comprise an $\F_q$-vector space. For the last statement, observe that $L_G(L_g(y))=y^{q^n}-y=0$ for every $y\in \F_{q^n}$, hence the image set $L_g(\F_{q^n})$ is contained in the set of the roots of $L_G(x)$. From the Rank-Nullity Theorem and the fact that $g(x)$ divides $x^n-1$, $L_g(\F_{q^n})$ has $q^{n-\deg(g)}=\deg(L_G)$ elements, from where the result follows.
\end{proof}

We obtain the following result.

\begin{proposition}\label{prop:values}
Let $f, g\in \F_q[x]$, and let  $V$ be the set of distinct roots of $L_f(x)=0$ over $\F_{q^n}$. Then the map $c\mapsto L_g(c)$ is a one to one correspondence  of $V$ with itself if and only if $\gcd(f(x), g(x), x^n-1) =1$. 
\end{proposition}

\begin{proof}
First, we show that $L_g(V)\subseteq V$. In fact, for $v\in V$, we have that $L_f(L_g(v))=L_{fg}(v)=L_g(L_f(v))=L_g(0)=0$ and so $L_g(v)\in V$.  Therefore, from Corollary~\ref{cor:roots} and the fact that $y\mapsto L_g(y)$ is $\F_q$-linear, the following are equivalent:

\begin{enumerate}[(i)]
\item the map $c\mapsto L_g(c)$ is a one to one correspondence  of $V$ with itself;
\item the only root of $L_g(x)$ lying in $V$ is $x=0$;
\item the only common root of $L_g(x)$ and $L_f(x)$ lying in $\F_{q^n}$ is $x=0$;
\item $\gcd(f(x), g(x), x^n-1)=1$.
\end{enumerate}
\end{proof}

\subsection{A permutation criterion}
We observe that a polynomial like in Eq.~\eqref{eq:main} can be written as $f(L_g(x))+k(L_g(x))\cdot L_h(x)$, where $g, h\in \F_q[x]$ are polynomials such that $g(x)$ divides $x^n-1$. The following theorem provides a general permutation criterion on when such a polynomial permutes $\F_{q^n}$ and it is the starting point of our results.

\begin{theorem}\label{agw-linearized}
Let $P\in \F_{q^n}[x]$ be such that $P(x)=f(L_g(x))+k(L_g(x))\cdot L_h(x)$, where $g, h\in \F_q[x]$, $g(x)$ is a monic divisor of $x^n-1$ and $k(L_g(\F_{q^n}))\subseteq \F_q^*$. Then $P$ is a PP of $\F_{q^n}$ if and only if the following hold:

\begin{enumerate}[(i)]
\item   $\gcd(g, h)=1$; 
\item the polynomial $Q(x)=L_g(f(x))+k(x)\cdot L_h(x)$ permutes the set $L_g(\F_{q^n})$.
\end{enumerate}
\end{theorem}

\begin{proof}
We employ a special case of the AGW criterion~\cite{AGW}. In the notation of Theorem~1.5 in~\cite{AGW}, take $\Psi=\overline{\Psi}=L_g$ and $\varphi=L_h$. Then $P$ permutes $\F_{q^n}$ if and only if $Q(x)=L_g(f(x))+k(x)\cdot L_h(x)$ permutes the set $L_g(\F_{q^n})$ and $0\in \F_{q^n}$ is the only common root of $L_g$ and $L_h$, defined over $\F_{q^n}$. Since $g$ divides $x^n-1$, Proposition~\ref{prop:values} entails that the latter holds if and only if $\gcd(g, h)=1$.
\end{proof}

We observe that the previous theorem is not effective if $g(x)= 1$. So we may require that $g(x)$ is not trivial, i.e., we are looking for polynomials with $g$ of positive degree (preferably high). The following result provides some PP's arising from a generic divisor of $x^n-1$.

 \begin{proposition}\label{prop:generic-g}
 Let $P\in \F_{q^n}[x]$ be such that $P(x)=f(L_g(x))+L_h(x)$, where $g, h\in \F_q[x]$ and $g(x)$ is a divisor of $x^n-1$.  Suppose that $L_g(f(L_g(y))) = 0$ for every $y\in \F_{q^n}$, i.e., the image of the set $L_g(\F_{q^n})$ by $f(x)$ is contained in the kernel of $L_g$. Then $P$ is a PP of $\F_{q^n}$ if and only if  $\gcd(x^n-1, h(x))=1$. The former holds if, for instance, $f$ is of the form $L_{G}(f_0(x))$ with $f_0\in \F_{q^n}[x]$ and $G(x)=\frac{x^n-1}{g(x)}$. 
 \end{proposition}
  
\begin{proof}
We employ Theorem~\ref{agw-linearized} with $k(x)=1$. From hypothesis, $Q(z)=L_g(f(z))+L_h(z)=L_h(z)$ for every $z\in L_g(\F_{q^n})$. Hence Theorem~\ref{agw-linearized} implies that $P$ is a PP of $\F_{q^n}$ if and only if $\gcd(g, h)=1$ and $L_h(x)$ permutes the set $L_g(\F_{q^n})$. From Corollary~\ref{cor:roots}, $L_g(\F_{q^n})$ equals the set of distinct roots of $L_{g_0}(x)$, where $g_0(x)=\frac{x^n-1}{g(x)}$ and then, by Proposition~\ref{prop:values}, the former holds if and only if $\gcd(g_0, h)=1$. In conclusion, $P$ is a PP of $\F_{q^n}$ if and only if $\gcd(g, h)=\gcd(g_0, h)=1$. Since $g(x)\cdot g_0(x)=x^n-1$, the latter is equivalent  to $\gcd(x^n-1, h(x))=1$.

\end{proof}

The following proposition provides a family of PP's that arise from Theorem~\ref{agw-linearized} with a special restriction on the polynomials $g$ and $h$.

\begin{proposition}\label{prop:PP's}
Let $g, h\in \F_q[x]$ be relatively prime polynomials such that $g(x)$ divides $x^n-1$ and $g(x)\cdot h(x)$ is divisible by $x^n-1$.  For every polynomial $f\in \F_{q^n}[x]$,  $P_{f, g, h}(x):=f(L_g(x))+L_h(x)$ is a PP of $\F_{q^n}$ if and only if $L_g(f(x))$ permutes the set $L_g(\F_{q^n})$. In particular, if $f(L_g(\F_{q^n}))\subseteq L_g(\F_{q^n})$, the latter holds if and only if $f$ permutes the set $L_g(\F_{q^n})$.
\end{proposition}

\begin{proof}
The first part follows directly from Theorem~\ref{agw-linearized} since $L_{gh}(y)=L_{x^n-1}(y)=y^{q^n}-y=0$ for every $y\in \F_{q^n}$. For the second part, it suffices to prove that $L_g$ permutes the set $L_g(\F_{q^n})$. From Corollary~\ref{cor:roots}, $L_g(\F_{q^n})$ is just the set of the roots of the polynomial $L_{g_0}(x)$, where $g_0(x)=\frac{x^n-1}{g(x)}$. Since $x^n-1$ divides $g(x)\cdot h(x)$, it follows that $g_0$ divides $h$ and, since $\gcd(g, h)=1$, we conclude that $\gcd(g, g_0)=1$. The latter combined with Proposition~\ref{prop:values} implies that $L_g$ permutes the set $L_g(\F_{q^n})$. 
\end{proof}

\section{The case $g(x)=\frac{x^n-1}{x-1}$} 

Here we focus on PP's arising from the factor $g(x)=\frac{x^n-1}{x-1}$. In other words, we are considering permutation polynomials of the form  $$f(\Tr_{q^n/q}(x))+k(\Tr_{q^n/q}(x))\cdot L_h(x),$$ where $k, h\in \F_{q}[x]$, $f\in \F_{q^n}[x]$ and $k(\F_{q})\subseteq \F_q^*$. The following definition is frequently used.

\begin{definition}
For a polynomial $f(x)=\sum_{i=0}^da_ix^i\in \F_{q^n}[x]$, we set $$T_n[f](x)=\sum_{i=0}^d\Tr_{q^n/q}(a_i)\cdot x^i\in \F_q[x].$$ 
\end{definition}

We obtain the following result.

\begin{theorem}\label{thm:trace}
The polynomial $P(x)=f(\Tr_{q^n/q}(x))+k(\Tr_{q^n/q}(x))\cdot L_h(x)\in \F_{q^n}[x]$ with $h\in \F_q[x]$ and $k(\F_{q})\subseteq \F_q^*$ is a PP over $\F_{q^n}$ if and only if the following conditions hold:

\begin{enumerate}
\item $\gcd\left(h(x), \frac{x^n-1}{x-1}\right)=1$; 
\item $Q(x):=T_n[f](x)+k(x)\cdot h(1)\cdot  x\in \F_q[x]$ is a PP over $\F_q$.
\end{enumerate}
In affirmative case, if $R$ is the inverse PP of $Q$ over $\F_q$, then the inverse PP of $P$ over $\F_{q^n}$ equals $$P_0(x)=F(\Tr_{q^n/q}(x))+k(R( \Tr_{q^n/q}(x)))^{q-2}\cdot L_H(x),$$ where $H\in \F_q[x]$ and $F\in \F_{q^n}[x]$ are given as follows:
\begin{enumerate}[(i)]
\item if $p\mid n$, then $H\in \F_q[x]$ is the unique polynomial of degree at most $n-1$ such that $h(x)\cdot H(x)\equiv 1\pmod {x^n-1}$ and $F$ is any polynomial satisfying $F(x)\equiv -k(R(x))^{q-2}\cdot L_H(f(R(x)))\pmod {x^{q}-x}$;
\item if $p\nmid n$, then $H\in \F_q[x]$ is the unique polynomial of degree at most $n-2$ such that $h(x)\cdot H(x)\equiv 1\pmod {\frac{x^n-1}{x-1}}$ and $F$ is any polynomial satisfying $F(x)\equiv M(R(x))\pmod {x^q-x}$, where $$M(x)=-k(x)^{q-2}\cdot L_H(f(x))+x\cdot \frac{1-h(1)\cdot H(1)}{n}.$$
\end{enumerate}
We remark that the polynomial $H$ in case (i) always exists since the condition $\gcd\left(h(x), \frac{x^n-1}{x-1}\right)=1$ is equivalent to $\gcd(h(x), x^n-1)=1$ if $p \mid n$.
\end{theorem}

\begin{proof}
Observe that, for $a\in \F_q$, $\Tr_{q^n/q}(f(a))=T_n[f](a)$. In particular, $\Tr_{q^n/q}(f(y))+k(y)\cdot L_h(y)=Q(y)$ for every $y\in \F_q$. Since $\Tr_{q^n/q}(\F_{q^n})=\F_q$, 
the permutation criterion on $P$ follows by Theorem~\ref{agw-linearized}. It remains to prove the statement regarding the inverse of $P$. For $a\in \F_{q^n}$, set $e_{a}=\Tr_{q^n/q}(a)\in \F_q$, hence $P(a)=f(e_a)+k(e_a)\cdot L_h(a)$ and so $\Tr_{q^n/q}(P(a))=Q(e_a)$ for every $a\in \F_{q^n}$. We obtain that
\begin{align*}P_0(P(a))& =F(\Tr_{q^n/q}(P(a))+k(R(\Tr_{q^n/q}(P(a))))^{q-2}\cdot L_H(P(a))\\ {} &= F(Q(e_a))+k(R(Q(e_a)))^{q-2}\cdot L_H(f(e_a)+k(e_a)\cdot L_h(a))\\ {}& = F(Q(e_a))+k(e_a)^{q-2}\cdot (L_H(f(e_a))+k(e_a)\cdot L_{hH}(a))\\ {}& = F(Q(e_a))+k(e_a)^{q-2}\cdot L_H(f(e_a))+ L_{hH}(a),\end{align*}
where in the last equality we used the fact that $k(\F_q)\subseteq \F_q^*$, i.e., $k(e_a)^{q-1}=1$. We split the proof into the cases $p\mid n$ and $p\nmid n$.

\begin{enumerate}[(i)]
\item If $p\mid n$ and $F(x)\equiv -k(R(x))^{q-2}\cdot L_H(f(R(x)))\pmod{x^q-x}$, we obtain that 
\begin{align*}P_0(P(a))&=-k(R(Q(e_a)))^{q-2}\cdot L_H(f(R(Q(e_{a}))))+k(e_a)^{q-2}\cdot L_H(f(e_{a}))+L_{hH}(a)\\{} & =-k(e_a)^{q-2}\cdot L_H(f(e_a))+k(e_a)^{q-2}\cdot L_H(f(e_{a}))+L_{hH}(a)=a ,\end{align*}
where in the last equality we used the fact that $h(x)\cdot H(x)\equiv 1\pmod {x^n-1}$.

\item If $p\nmid n$, from $\Tr_{q^n/q}(1)=n\ne 0\in \F_q$, it follows that $a=e_a/n+\delta$ for some $\delta\in \F_{q^n}$ with $\Tr_{q^n/q}(\delta)=0$. Since $h(x)\cdot H(x)\equiv 1\pmod {\frac{x^n-1}{x-1}}$, we have that $$L_{hH}(a)=L_{hH}(\delta+e_a/n)=\delta+L_{hH}(e_a/n)=\delta+e_a\cdot H(1)\cdot h(1)/n.$$
Therefore, 
$$P_0(P(a))=F(Q(e_a))+k(e_a)^{q-2}\cdot L_H(f(e_a))+\delta+e_a\cdot H(1)\cdot h(1)/n.$$
Since $F(x)\equiv M(R(x))\pmod {x^q-x}$ and $M(x)=-k(x)^{q-2}\cdot L_H(f(x))+x\cdot \frac{1-h(1)\cdot H(1)}{n}$, we have that $$P_0(P(a))=M(e_a)+k(e_a)^{q-2}\cdot L_H(f(e_a))+\delta+e_a\cdot \frac{H(1)\cdot h(1)}{n}=\delta+\frac{e_a}{n}=a.$$

\end{enumerate}
\end{proof}

If $f, h\in \F_q[x]$, then $L_h(f(x))\equiv h(1)\cdot f(x)\pmod {x^q-x}$. In particular, Theorem~\ref{thm:trace} readily implies the following corollary. 

\begin{corollary}\label{cor:1}
The polynomial $P(x)=f(\Tr_{q^n/q}(x))+k(\Tr_{q^n/q}(x))\cdot L_h(x)$ with $f, h\in \F_q[x]$ and $k(\F_q)\subseteq \F_q^*$ is a $PP$ over $\F_{q^n}$ if and only if $\gcd\left(h(x), \frac{x^n-1}{x-1}\right)=1$ and $Q(x)=n\cdot f(x)+h(1)k(x)\cdot x$ induces a permutation of $\F_q$.
More specifically,  $P(x)=f(\Tr_{q^n/q}(x))+k(\Tr_{q^n/q}(x))\cdot L_h(x)$ is a PP of $\F_{q^n}$ if and only if one of the following holds: 
\begin{enumerate}[(i)]
\item  $p \mid n$, $\gcd(h(x), x^n-1)=1$ and $k(x)\cdot x\in \F_q[x]$ is a PP over $\F_q$;
\item $p\nmid n$,  $\gcd\left(h(x), x^n-1\right)=x-1$ and $f\in \F_q[x]$ is a PP over $\F_q$;
\item $p\nmid n$,  $\gcd\left(h(x), x^n-1\right)=1$ and $nf(x)+h(1)\cdot k(x)\cdot x$ is a PP over $\F_q$.
\end{enumerate}
\end{corollary}




Recall that $b\in \F_q[x]$ is a complete permutation polynomial (CPP) or  complete mapping over $\F_q$ if both $b(x)$ and $b(x)+x$ are permutations of $\F_q$. We obtain the following corollary.

\begin{corollary}\label{cor:cpp}
Let $P(x)=f(\Tr_{q^n/q}(x))+L_h(x)\in \F_{q^n}[x]$ be a polynomial such that $h\in \F_q[x]$. Then $P(x)$ is a CPP over $\F_{q^n}$ if and only if the following hold:

\begin{enumerate}[(i)]
\item $\gcd\left(h(x)\cdot (h(x)+1), \frac{x^n-1}{x-1}\right)=1$; 
\item $Q(x):=T_n[f](x)+h(1)\cdot x\in \F_q[x]$ is a CPP over $\F_q$.
\end{enumerate}
\end{corollary}

\begin{proof}
We observe that if $h_0(x)=h(x)+1$, then $P(x)+x=f(\Tr_{q^n/q}(x))+L_{h_0}(x)$ and $T_n[f](x)+h_0(1)\cdot x=Q(x)+x$. The result follows from Theorem~\ref{thm:trace} by taking $k(x)=1$.
\end{proof}

\subsection{From PP's of $\F_q$ to PP's of $\F_{q^n}$}
From Theorem~\ref{thm:trace}, we obtain the following method for producing permutations of the extension field $\F_{q^n}$ from permutations of the base field $\F_q$. 

\begin{theorem}\label{thm:const}
Let $b\in \F_q[x]$ be a permutation polynomial, let $n$ be a positive integer and let $h, k\in \F_q[x]$ be polynomials such that $\gcd\left(h(x), \frac{x^n-1}{x-1}\right) =1$ and $k(\F_q)\subseteq \F_q^*$. Then for every polynomial $f\in \F_{q^n}[x]$ such that $T_n[f](x)\equiv b(x)-h(1)\cdot k(x)\cdot x\pmod {x^q-x}$, the polynomial
$$\mathcal P_{b, f, h, k, n}(x):=f(\Tr_{q^n/q}(x))+k(\Tr_{q^n/q}(x))\cdot L_h(x),$$
permutes $\F_{q^n}$. 
\end{theorem}

In simpler terms, Theorem~\ref{thm:const} provides the following procedure for constructing permutations of $\F_{q^n}$.
\begin{algorithm}
\begin{itemize}
\item Step 1: pick a permutation polynomial $b \in \F_q[x]$. \\
\item Step 2: pick polynomials $h, k\in \F_{q}[x]$ such that  $\gcd\left(h(x), \frac{x^n-1}{x-1}\right) =1$ and $k(\F_q)\subseteq \F_q^*$.
\item Step 3: find $f\in \F_{q^n}[x]$ such that $$T_n[f](x)\equiv b(x) - k(x)\cdot h(1)\cdot x\pmod {x^q-x}.$$ 
\item Step 4: construct the permutation polynomial $$\mathcal P_{b, f, h, k, n} (x) = f(\Tr_{q^n/q}(x))+k(\Tr_{q^n/q}(x))\cdot L_h(x).$$ 
\end{itemize}
\end{algorithm}

We observe that the functional equation $$T_n[f](x)\equiv b(x)-k(x)\cdot h(1)\cdot x\pmod {x^q-x},$$ with the restriction $\deg(f)<q$ gives rise to a linear system of $q$ equations involving trace functions. More specifically, if we write $f(x)=\sum_{i=0}^{q-1}y_ix^i, b(x)=\sum_{i=0}^{q-1}b_ix^i$ and $k(x)=\sum_{i=0}^{q-1}k_ix^i$, the functional equation is equivalent to the following system of equations:
\begin{align*}
\Tr_{q^n/q}(y_0) &= b_0,\\
\Tr_{q^n/q}(y_1) &= b_1 - k_0\cdot h(1)-k_{q-1}\cdot h(1),  \\
\Tr_{q^n/q}(y_i) &= b_i-k_{i-1}\cdot h(1), \; 2\le i\le n.
\end{align*}
In particular, for fixed $n, b, h$ and $k$, there are $q^{(n-1)q}$ distinct choices for $f$ under the condition $\deg(f)<q$. Moreover, such condition implies that the corresponding polynomials $\mathcal P_{b, f, h,k,  n}$ are all incongruent modulo $x^{q^n}-x$. Hence we obtain $q^{(n-1)q}$ distinct PP's of $\F_{q^n}$.

\begin{remark}\label{rem:main}
The following are easily verified.
\begin{enumerate}[(i)]

\item We may also iterate the procedure above, hence obtaining a sequence $\{\mathcal P_i\}_{i\ge 1}$ of polynomials such that $\mathcal P_i$ permutes $\F_{q^{n^i}}$ for every $i\ge 1$.

\item In the context of Theorem~\ref{thm:const}, if we also know the inverse of the permutation $b(x)$, we may employ Theorem~\ref{thm:trace} and obtain the inverse PP of any permutation $\mathcal P_{b, f, h, k, n}$. 

\item From Corollary~\ref{cor:cpp}, the method above can be extended to the construction of complete permutation polynomials of $\F_{q^n}$ with $k(x)=1$: we start with $b \in \F_q[x]$, a CPP of $\F_q$, and pick $h\in \F_q[x]$ such that  $\gcd\left(h(x)\cdot (h(x)+1), \frac{x^n-1}{x-1}\right)=1$. Similarly,  from such construction, we obtain $q^{(n-1)q}$ distinct CPP's of $\F_{q^n}$.
\end{enumerate}
\end{remark}

The following result exemplifies the applicability of Theorem~\ref{thm:const} to the explicit construction of permutations and complete permutation polynomials of $\F_{q^n}$ with $k(x)=1$. Its proof follows directly by Theorem~\ref{thm:const} and the fact that $\Tr_{q^n/q}(\theta^q-\theta)=0$ for every $\theta\in \F_{q^n}$. 

\begin{corollary}
Let $n$ be a positive integer, not divisible by the characteristic $p$ of $\F_q$. Let $h\in \F_q[x]$ be a polynomial such that  $\gcd\left(h(x), \frac{x^n-1}{x-1}\right) =1$ and let $b(x) = \sum_{i=0}^{m} b_i x^i$ be any permutation polynomial of $\fq$.   For every $\theta_0, \ldots,  \theta_{q-1}\in \F_{q^n}$,    the polynomial 
\begin{equation*}
P(x)=L_h(x)- \frac{h(1)}{n}  \cdot \Tr_{q^n/q}(x)+ \displaystyle{\sum_{i=0}^{m }\frac{b_i}{n}\cdot \Tr_{q^n/q}(x)^i + \sum_{i=0}^{q-1}(\theta_i^q-\theta_i )\cdot \Tr_{q^n/q}(x)^i },
\end{equation*}
permutes $\F_{q^n}$. In addition, if $\gcd\left(h(x)+1, \frac{x^n-1}{x-1}\right)=1$ and $b(x)$ is a CPP of $\F_q$, then $P(x)$ is also a CPP of $\F_{q^n}$.
\end{corollary}



For a permutation monomial $b(x) = ax^m$ of $\fq$, Theorem~\ref{thm:const} immediately implies the following result. 

\begin{corollary}
Let $h\in \F_q[x]$ be a polynomial such that  $\gcd\left(h(x), \frac{x^n-1}{x-1}\right) =1$ and let $m$ be a positive integer such that $m<q-1$ and $\gcd(m, q-1)=1$. The following hold:
\begin{enumerate}[(i)]
\item if $m=1$, then for every $\alpha, \theta_0, \ldots, \theta_{q-1}\in \F_{q^n}$  such that $\Tr_{q^n/q}(\alpha)\ne -h(1)$, the polynomial $$P(x)=L_h(x)+\alpha\cdot \Tr_{q^n/q}(x)+\sum_{i=0}^{q-1}(\theta_i^q-\theta_i)\cdot \Tr_{q^n/q}(x)^i,$$ permutes $\F_{q^n}$. In particular, if $n$ is not divisible by the characteristic $p$ of $\F_q$, one may take $\alpha$ as any element in $\F_q\setminus\{-\frac{h(1)}{n}\}$.

\item if $m>1$, then for every $\alpha, \beta, \theta_0, \ldots, \theta_{q-1}\in \F_{q^n}$ such that $\Tr_{q^n/q}(\alpha)=-h(1)$ and  $\Tr_{q^n/q}(\beta)\ne 0$, the polynomial $$P(x)=L_h(x)+\alpha\cdot \Tr_{q^n/q}(x)+\beta\cdot \Tr_{q^n/q}(x)^m+\sum_{i=0}^{q-1}(\theta_i^q-\theta_i)\cdot \Tr_{q^n/q}(x)^i,$$ permutes $\F_{q^n}$. In particular, if $n$ is not divisible by the characteristic $p$ of $\F_q$, one may take $\alpha=\frac{-h(1)}{n}$ and $\beta$ as any nonzero element of $\F_q$. 

\end{enumerate}
\end{corollary}




\subsection{Extension to generic divisors of degree $n-1$.}

The previous result can be easily extended to any other divisor of $x^n-1$ over $\F_q$ with degree $n-1$. If $a\in \F_{q}$ is a root of the polynomial $x^n-1$ (this holds for some $a\ne 1$ if $\gcd(q-1, n)>1$), and $g_a(x)=\frac{x^n-1}{x-a} = x^{n-1} + ax^{n-2} + \cdots + a^{n-1}$, then we can obtain an analogue of Theorem~\ref{thm:trace}.  In this case,  $L_{g_a}(x) = x^{q^{n-1}} + ax^{q^{n-2}} + \cdots + a^{n-1} x$ and $L_{g_a}(\fqn) = \delta \cdot \fq$, where $\delta \in \fqn$ any element satisfying  $\delta^q = a \delta$.  In this analogy,
we replace conditions (1) and (2) there by $\gcd\left(\frac{x^n-1}{x-a}, h(x)\right)=1$ and $Q_a(x)=T_n^{(a)}[f](x)+k(x)\cdot h(a)\cdot x$  permutes the set $\delta \cdot \F_{q}$,  where $\delta\in \F_{q^n}$ is any element with $\delta^q=a\delta$ and for $f(x)=\sum_{i=0}^ma_ix^i\in \F_{q^n}$, 
$$T_n^{(a)}[f](x)=a^{-1}\sum_{i=0}^m\left(\sum_{j=0}^{n-1}a^{(i-1)j}\cdot a_i^{q^j}\right)x^i=a^{-1}\sum_{i=0}^m\left(\sum_{j=0}^{n-1}\delta^{(i-1)(q^j-1)}\cdot a_i^{q^j}\right)x^i.$$ 
We observe that $Q_a(x)$ permutes $\delta \cdot \F_{q}$ if and only if $\delta^{-1}Q_a(\delta x)$ permutes $\F_q$.  By a direct computation, we obtain the following polynomial:
$$ \bar{Q}_a(x) = \delta^{-1}Q_a(\delta x)  = \frac{1}{a}  \sum_{i=0}^m  \Tr(\delta^{i-1} a_i)  x^i + k(\delta x) h(a) x. $$ 

In terms of AGW criterion, we have the following commutative diagram. 

\[
\xymatrixcolsep{12pc}
\xymatrix{
\fqn  \ar[r]^{P(x) = f(L_{g_a}(x) )+k(L_{g_a}(x))\cdot L_h(x)     } \ar[d]_{L_{g_a}(x) } & \fqn \ar[d]^{L_{g_a}(x) }\\
\delta\cdot \fq  \ar[r]^{Q_a(x)=T_n^{(a)}[f](x)+k(x)\cdot h(a)\cdot x}  \ar[d]_{\delta^{-1} x } & \delta\cdot \fq  \ar[d]^{ \delta^{-1}x}\   \\
\fq  \ar[r]^{\bar{Q}_a(x)=\frac{1}{a}  \sum_{i=0}^m  \Tr(\delta^{i-1} a_i)  x^i + k(\delta x) h(a) x} &  \fq 
}
\]

We summarize our result as follows:

\begin{theorem}\label{thm:tracevariant}
Let $a \in \fq$ such that $a^n=1$ and $\delta \in \fqn$ such that $\delta^q = a \delta$. The polynomial $P(x)=f(L_{g_a}(x) )+k(L_{g_a}(x))\cdot L_h(x)\in \F_{q^n}[x]$ with  $L_{g_a}(x) = x^{q^{n-1}} + ax^{q^{n-2}} + \cdots + a^{n-1} x$,  $f(x) = \sum_{i=0}^m a_i x^i$, $h(x)\in \F_q[x]$  and  $k(\delta \F_{q})\subseteq \F_q^*$ is a PP over $\F_{q^n}$ if and only if the following conditions hold:

\begin{enumerate}
\item $\gcd\left(h(x), \frac{x^n-1}{x-a}\right)=1$; 
\item $\bar{Q}_a(x)  = \frac{1}{a}  \sum_{i=0}^m  \Tr(\delta^{j-1} a_j)  x^i + k(\delta x) h(a) x$ is a PP over $\F_q$.
\end{enumerate}
In affirmative case, if $R$ is the inverse PP of $\bar{Q}_a$ over $\F_q$, then the inverse PP of $P$ over $\F_{q^n}$ equals $$P_0(x)=F(\delta^{-1} L_{g_a}(x))+k(\delta R(\delta^{-1} L_{g_a}(x)))^{q-2}\cdot L_H(x),$$ where $H\in \F_q[x]$ and $F\in \F_{q^n}[x]$ are given as follows:
\begin{enumerate}[(i)]
\item if $p\mid n$, then $H\in \F_q[x]$ is the unique polynomial of degree at most $n-1$ such that $h(x)\cdot H(x)\equiv 1\pmod {x^n-1}$ and $F$ is any polynomial satisfying $F(x)\equiv -k(\delta R(x))^{q-2}\cdot L_H(f(\delta R(x)))\pmod {x^{q}-x}$;
\item if $p\nmid n$, then $H\in \F_q[x]$ is the unique polynomial of degree at most $n-2$ such that $h(x)\cdot H(x)\equiv 1\pmod {\frac{x^n-1}{x-a}}$ and $F$ is any polynomial satisfying $F(x)\equiv M(\delta R(x))\pmod {x^q-x}$, where $$M(x)=-k(x)^{q-2}\cdot L_H(f(x))+ax\cdot \frac{1-h(a)\cdot H(a)}{n}.$$
\end{enumerate}
We remark that the polynomial $H$ in case (i) always exists since the condition $\gcd\left(h(x), \frac{x^n-1}{x-a}\right)=1$ is equivalent to $\gcd(h(x), x^n-1)=1$ if $p\mid n$.
 
 \begin{proof}
 In this case, $L_{g_a}(\delta) = \delta^{q^{n-1}} + a \delta^{q^{n-2}}  + \cdots + a^{n-1} \delta = n a^{n-1} \delta=na^{-1}\delta$. For each $b\in \fqn$, we set $e_b = L_{g_a}(b) \in \delta\cdot \fq$. In particular, 
 $b=  a\cdot \frac{e_{b}}{n} + \alpha$ for some $\alpha$ such that $L_g(\alpha) =0$.   The rest of  proof follows as the proof of Theorem~\ref{thm:trace}. 
 \end{proof}
 
\end{theorem}

Similarly, we can recursively construct permutation polynomials of $\fqn$ from a permutation polynomial  $b(x)=\sum_{i=0}^{q-1}b_ix^i$ of $\fq$.   Let $k(x)=\sum_{i=0}^{q-1}k_ix^i\in \F_{q^n}[x]$ such that $k(\delta x) \in \fq[x]$ and $k(\delta \fq) \subseteq \fq^*$. 
 
From the functional equation $\bar{Q}_a(x)  \equiv b(x) \pmod {x^q-x}$, we obtain the following system of equations:
\begin{align*}
 \Tr_{q^n/q}(\delta^{-1} y_0) &=a b_0,\\
  \Tr_{q^n/q}(y_1) &= a(b_1 - k_0\cdot h(a)-k_{q-1}\delta^{q-1} \cdot h(a)),  \\
 \Tr_{q^n/q}(\delta^{i-1} y_i) &= a(b_i-k_{i-1} \delta^{i-1} \cdot h(a)), \; 2\le i\le q-1.
\end{align*}
In particular, for fixed $n, b, h$ and $k$, there are $q^{(n-1)q}$ distinct choices for $f$ under the condition $\deg(f)<q$. Moreover, such condition implies that the corresponding polynomials  are all incongruent modulo $x^{q^n}-x$. Hence we obtain $q^{(n-1)q}$ distinct PP's of $\F_{q^n}$.

\section{Conclusion and future research}
In this paper we have explored the permutational property of polynomials of the form $f(L_g(x))+k(L_g(x))\cdot L_h(x)\in \F_{q^n}[x]$ over $\F_{q^n}$, where $L_g, L_h$ are the $q$-associates of polynomials $g, h\in \F_q[x]$ with $g(x)$ a divisor of $x^n-1$ and $k\in \F_q[x]$ satisfies $k( \delta \F_q)\subseteq \F_q^*$. With the help of the AGW criterion, we have provided some general results on PP's of this form, relating to arithmetic properties of the polynomials $g$ and $h$, like the condition $\gcd(g, h)=1$. By specializing to the case when the divisor $g(x)$  is of degree $n-1$, we obtained results on the construction of PP's, CPP's and their inverses. 

It would be interesting to explore the construction of PP's arising from other factors of $x^n-1$. The natural next step is to consider divisors of $x^n-1$ of degree $n-2$ over $\F_q$, where we have to discuss the permutational property of polynomials over $\F_q$-vector spaces of dimension $2$.


\newcommand{\etalchar}[1]{$^{#1}$}

\end{document}